\documentclass[reqno]{amsart}
\usepackage{amssymb}
\usepackage[matrix,arrow,tips,curve,ps]{xy}
\input xypic
\usepackage{comment}
\SelectTips{cm}{10}
\usepackage{enumerate}
\usepackage{color}
\newtheorem{theorem}{Theorem}[section]
\newtheorem{lemma}[theorem]{Lemma}
\newtheorem{proposition}[theorem]{Proposition}
\newtheorem{corollary}[theorem]{Corollary}
\newtheorem{conjecture}[theorem]{Conjecture}
\newtheorem{example}[theorem]{Example}
\newtheorem{remarks}[theorem]{Remarks}

\theoremstyle{remark}
\newtheorem{claim}[theorem]{Claim}
\newtheorem{remark}[theorem]{Remark}

\numberwithin{equation}{section}

\newcommand{\PP}{\mathbb{P}}
\newcommand{\AAA}{\mathbb{A}}
\newcommand{\C}{\mathbb{C}}
\newcommand{\G}{\mathbb{G}}
\newcommand{\Q}{\mathbb{Q}}
\renewcommand{\le}{\leqslant}
\renewcommand{\ge}{\geqslant}
\newcommand{\SSS}{\mathcal{S}}
\newcommand{\TTT}{\mathcal{T}}
\newcommand{\I}{\mathcal{I}}

\renewcommand{\cL}{\mathcal{L}}

\newcommand{\cO}{\mathcal{O}}
\newcommand{\cP}{\mathcal{P}}
\newcommand{\cX}{\mathcal{X}}
\newcommand{\cZ}{\mathcal{Z}}

\DeclareMathOperator{\Proj}{\mathbb P}
\DeclareMathOperator{\p}{\mathbb P}
\DeclareMathOperator{\Sec}{{\rm Sec}}
\DeclareMathOperator{\Sing}{{\rm Sing}}
\DeclareMathOperator{\Tan}{{\rm Tan}}
\DeclareMathOperator{\tildetau}{\tilde \tau}
\DeclareMathOperator{\II}{{\rm II}}
\DeclareMathOperator{\map}{\dasharrow}
\DeclareMathOperator{\Bl}{{\rm Bl}}

\begin{document}
\title[On the classification of OADP varieties]
{On the classification of OADP varieties}

\author{Ciro Ciliberto}

\address{Dipartimento di Matematica,
Universit\`a di Roma Tor Vergata,
Via della Ricerca Scientifica, 00133 Roma, Italy}
\email{cilibert@mat.uniroma2.it}

\author{Francesco Russo}

\address{Dipartimento di Matematica,
Universit\`a di Catania,
Viale A. Doria 6, 95125 Catania, Italy}
\email{frusso@dmi.unict.it}

\begin{abstract} 
The main purpose of this paper is to show that OADP varieties
stand at an important crossroad of various main streets in 
different disciplines like projective geometry,
birational geometry and algebra. This is a good reason for studying and classifying them. 
Main specific results are: 
(a) the classification
of all OADP surfaces (regardless to their smoothness); (b) the classification
of a relevant class of normal OADP varieties of any dimension, which includes interesting
examples like lagrangian grassmannians. Following \cite {PR}, the equivalence of the classification in (b) 
with the one  of quadro--quadric Cremona transformations and of 
complex, unitary, cubic Jordan algebras are explained. 
\end{abstract}

\maketitle

\section*{Introduction}
The number of {\it apparent double points} of
 an irreducible projective variety $X$ of dimension $n$ in
$\Proj^{2n+1}$ is the number of secant lines to $X$ passing through a  general
point of $\Proj^{2n+1}$. The variety $X$ is called a {\it variety with one
apparent double point}, or {OADP} variety, if this number is 1. 
Hence {OADP} varieties are not secant defective, and can be
regarded as {\it the simplest non defective
varieties} of dimension $n$ in $\Proj^{2n+1}$. Probably this is the
reason why {OADP} varieties early attracted
the attention of algebraic geometers. There are other
reasons however. Indeed, as we will see in this paper, 
{some classes} of {OADP} varieties are 
interesting on their own, exhibiting strong unexpected relations with different important
subjects like the theory of unitary, complex, cubic \emph{Jordan algebras}
and the theory of {\it quadro--quadric} Cremona transformations 
(see \S \ref{sec:hypercase}, \cite {PR} and  \cite{PR2}).\par

The first instance of an OADP variety  is a rational
normal cubic curve in $\Proj^3$, which 
is the only {OADP} curve (see Proposition \ref{prop:curves}). 
Severi \cite {Se} took up the classification of  {OADP} surfaces. 
According to him the only
{OADP} surfaces with at most finitely many singularities are
quartic rational normal scrolls and (weak) del Pezzo quintics (see Theorem \ref
{thm:classif}). Severi's argument was affected by a gap, recently fixed in
 \cite {Ru}, see also \cite{OADP}. In \cite{OADP}
 there is also the description of large classes
of smooth OADP varieties of arbitrary large dimensions and 
general results about OADP varieties. The main result in \cite{OADP} is the classification of smooth
three dimensional OADP varieties. In \S \ref{sec:verra} below
we explain a construction of another large class of OADP varieties of any dimensions
 suggested by
A. Verra, which we dub {\it Verra varieties}. They are (in general singular, even non--normal) scrolls.

The classification of OADP varieties is in general quite hard and, as far as we know, no result
{is} known in arbitrary dimension, even under restrictive hypotheses like smoothness or considering only particular subclasses. However, as mentioned above,
a few general important properties of OADP varieties have been established
in \cite{OADP} (some of them generalized in \cite{CR} and \cite{IR} from
different viewpoints). For example, any OADP variety is rational
since it projects birationally from a general tangent space (see \cite[Corollary 4.2]{OADP}). 
We call \emph{Bronowski varieties} those varieties $X\subset \p^{2n+1}$
of dimension $n$ that, like OADP varieties, project birationally to $\p^n$ from a general tangent space.
According to an unproved assertion of Bronowski in \cite {Br}, which we call  \emph{Bronowski's conjecture} (see Conjecture \ref {conj:bron} below),
any Bronowski variety should be an OADP variety. 
A general tangent hyperplane section  of a Bronowski variety  of dimension $n$, which is not a scroll
over a curve, is irreducible and rational (see Proposition \ref{hyperlane}).
 Applying this to normal surfaces one
recovers Severi's classification  right
away (see Theorem \ref{thm:classif}). One of the main results of this paper
is the classification of all irreducible OADP surfaces, regardless to the dimension of their singular locus: they are either as in Severi's theorem or are Verra surfaces (see Theorem \ref{thm:classif}).

The degree of Verra varieties of any dimension $n\ge 2$ is unbounded while the classification in dimension
1 and 2 suggests that the degree of normal OADP varieties should be bounded by a function
depending on $n$, i.e. one could expect that there are
finitely many families of normal OADP varieties of a given dimension. 
The approach in the present paper suggests that
an {\it expected bound} for the degree of normal OADP varieties of dimension $n$ should be
$2^n+1$ (see \S \ref {ssec:fund}, in particular Lemma \ref {lem:3}). 
This is attained for $n=2$, not attained for $n=3$ under the smoothness
assumption but it could be attained in the singular case
(recently an interesting, unpublished, example with an intricate and fascinating projective and birational geometry has been suggested).

Motivated by these ideas and problems and by the results in \cite{PR} on varieties $X\subset\p^{2n+1}$ which are
\emph{$3$--covered by rational normal cubics} (i.e. there is  a rational normal cubic in $X$ passing through three general points of it), 
we have been lead to subdivide OADP varieties, or more generally Bronowski varieties, into classes according to the degree of the
so called {\it fundamental hypersurface $V\subset\p^n$ of $X$}. This is the image via the tangential
projection $\tau:X\map\p^n$  at a general point $x\in X$, of the exceptional divisor $E$ of the blow--up of $X$ at $x$ (see \S\ref{ssec:fund}
for definitions and details). Since $V$ is the image of $E$ via the map associated to the second fundamental
form of $X$ at $x$, $\bar \tau:E\map V$ is given by a linear system of quadrics on $E\cong \p^{n-1}$, hence  $1\leq\deg(V)\leq 2^{n-1}$, which points to the above conjectural bound.
>From this point of view the case $\deg(V)=1$, which we call the {\it hyperplane case}, is the {simplest}, and the one which gives rise to varieties with relatively small degrees. Therefore one may expect that in this case we find more and more basic examples than in others. 

Inspired by these considerations, we concentrate  in \S\ref{sec:hypercase}  
on this class of normal and geometrically linearly normal (see \S \ref {ssec:linorm}) Bronowski varieties (which include normal OADP varieties). They can naturally be subdivided in two subclasses, according to the fact that the map $\bar \tau:E\map V$ is either linear (case $(H1)$) or quadratic (case $(H2)$). In the former case the classification is easy: only rational normal scrolls belong to class $(H1)$. The case $(H2)$ is more complicated and, on the other hand, much more interesting. We first show that the map $\bar \tau:E\map V$ is quadro--quadric (see Proposition \ref {prop:22}). Then we prove our main result to the effect that in the $(H2)$ case the map
$\tau^{-1}:\p^n\map X$ is given by a $(2n+1)$--dimensional linear system of cubics having singular base points at the base locus scheme of $\bar\tau^{-1}:V\map E$ (see  Theorem \ref{th:hyperplanar}). This implies that  $X$ is $3$--covered by rational normal cubics in the sense of \cite{PR} and that 
$X$ is OADP (see again \cite {PR}), thus proving Bronowski's conjecture in this case. 
To the best of our knowledge, this is  the first general result concerning (what we consider to be) a meaningful class of OADP varieties of any dimension. 

In  \cite {PR} and \cite{PR2} several results are proved for irreducible, non--degenerate varieties $X\subset \p^ {2n+1}$ of dimension $n$ which are 3-covered by rational normal cubics.  In view of Theorem \ref {th:hyperplanar}, all of them can be applied to normal, geometrically linearly normal, Bronowski varieties $X\subset \PP^{2n+1}$ presenting case $(H2)$.  For the reader's convenience, we summarize these applications in \S  
\ref {ssec:examples}. First of all, the birational map
$\bar \tau:E\map V$ is \emph{involutory}, i.e. it coincides with its inverse up to linear transformations (see Corollary \ref {cor:cone}), thus launching a bridge between OADP varieties and the classification of involutory quadratic
Cremona transformations. In \cite{PR} the connection of the varieties treated there with 
the theory of Jordan algebras has been elucidated. Indeed, as a consequence of the results of  \cite {PR, PR2},  the classification of the following  items  turns out to be equivalent: non--degenerate varieties $X\subset \p^ {2n+1}$ of dimension $n$ which are 3-covered by rational normal cubics; involutory quadro--quadric Cremona transformations of $\p^ {n-1}$; complex, cubic, unitary Jordan algebras $\mathbb J$ of dimension  $n$, via the construction of the associated twisted cubic $X_\mathbb J$,  which turns out to be a variety $X\subset \p^ {2n+1}$ of dimension $n$,  3-covered by rational normal cubics (see \S \ref{Jalg} for details). In view of our Theorem \ref {th:hyperplanar},
normal, geometrically linearly normal, Bronowski varieties naturally fit in this framework. Moreover one may conjecture that every twisted cubic associated to a Jordan algebra as above is normal so that all these items should in fact be equivalent.

In this framework, smooth Bronowski varieties presenting case $(H2)$ correspond
 to semi--simple Jordan algebras and to Cremona transformations with smooth base locus. The first aspect is taken care of by  \cite[Theorem 5.7]{PR}. The second aspect basically by Ein--Shepherd-Barron's classification of simple quadro--quadric Cremona  transformations (see \cite {esb}). 
The third by Jordan--von Neumann--Wigner classification theorem of simple, complex, unitary Jordan algebras (see \cite[p. 49] {mccrimmon}).  As pointed out above, these three aspects are equivalent. 
The result is that the only smooth, linearly normal, Bronowski varieties presenting case $(H2)$ are the  \emph{lagrangian grassmannians} and the Segre embedding of $\p^ 1$ times a smooth quadric (see \S \ref {ssec:laggrass}). The  representation of lagrangian grassmannians in terms of linear systems of cubic hypersurfaces singular along degenerate \emph{Severi varieties} (see \cite {Zak2}) is explained 
in  \S  \ref {Jalg}.  In \S \ref {ssec:3-4} we mention 
the cases of dimension 3 and 4. 

In conclusion, we hope that this paper could set up a useful framework for attacking the general  classification of OADP varieties, or at least of important classes of them, pointing to a boundedness result as conjectured above. Moreover, proceedings further on this route, may lead to a proof of Bronowski's conjecture. 
\medskip

{\bf Acknowledgements}. It is a pleasure for the authors to dedicate this paper to their friend and colleague Fabrizio Catanese. 

We thank Luc Pirio for many useful remarks and comments which improved the exposition and 
the referee for pointing out a lot of misprints in a previous version.

\section{Preliminaries}

\subsection{The secant variety}\label{secsub}  Let $X\subset \PP^r$ be a projective scheme over $\mathbb C$.
We let $\langle X\rangle$ be the \emph{span} of $X$ in $\PP^r$. 
If $x$ is a point of  $X$, we denote by $T_{X,x}$ the
\emph{embedded projective tangent space} to $X$ at $x$ (see \cite [p. 181]  {H}, where the notation
$T_x(X)$ is used for it).
We denote by $X_2$ the Hilbert scheme of length 2 subschemes of $X$
and by $X_{2,r}$ the subscheme of $X_2$ parametrizing  reduced subschemes.
If $X$ is reduced, we use the common terminology and say that $X$ is a \emph{variety}. 

Let  $X\subset \PP^r$  be a variety. The {\it abstract secant
variety} $\SSS(X)$ of $X$ is the Zariski closure of the incidence correspondence
\[
\{(\xi,p)\in X_{2,r}\times \PP^{r}:   p\in \langle \xi\rangle\}.
\]
The image of the projection of
${\SSS}(X)$ to $\PP^{r}$ is the {\it
secant variety} $\Sec(X)$ of $X$.  Namely, $\Sec(X)$ is the Zariski closure of the union of all \emph{secant lines} to $X$, i.e.
lines spanned by distinct points of $X$. Lines in $\Sec(X)$ which are are flat limits of proper
secant lines but  not proper secant lines 
are called \emph{improper secant lines}. A \emph{secant line}
may thus be  proper or improper.

One has the projection map $p_X:
{\SSS}(X)\to \Sec(X)$. If $X$ is an irreducible variety of dimension $n$,
then $\SSS(X)$ is irreducible and $\dim({\SSS}(X))=2n+1$, 
therefore $\Sec(X)$ is irreducible and 
$\dim(\Sec(X))\le \min \{2n+1,r\}$. The variety $X$ is said to be
{\it secant defective} if the strict inequality holds.

If $X$ is an irreducible variety of dimension $n$, \emph{Terracini's Lemma} asserts that
$X$ is secant defective if and only if, for any pair $x_1,x_2$ of general points
of $X$, one has $\dim(T_{X,x_1}\cap T_{X,x_2})> \min\{2n -r, -1\}$.

If $X\subset \PP^r$ is a scheme,
we can more generally consider  the {\it abstract big secant
scheme} of $X$, i.e. the incidence correspondence
\[
\SSS_b(X):=\{(\xi,p)\in X_2\times \PP^{r}:  p\in \langle \xi\rangle\}.
\]
The scheme theoretical image of the projection of
${\SSS}_b(X)$ to $\PP^{r}$ is the {\it big
secant scheme} $\Sec_b(X)$ of $X$. If $X$ is a variety,
$\Sec(X)\subseteq \Sec_b(X)$, with equality if $X$ is smooth, but the inclusion is in general strict.

\subsection{The dual variety} Given an irreducible variety $X\subset \PP^ r$, the \emph{dual variety} $X^*$ of $X$ is the image in ${\PP^ r}^ *$ of the \emph{conormal scheme} (or variety) of $X$, i.e. the Zariski closure $\cP_X$  in $X\times{ \PP^ r}^ *$ of the irreducible scheme
\[
\{(x,H): x\in X\setminus\Sing(X), H\in{ \PP^ r}^ *, T_{X,x}\subset H\}
\]
of dimension $r-1$. Hence $\dim(X^ *)\le r-1$ and $X\subset\p^r$ is called \emph{dual defective}
if strict inequality holds. We call $\epsilon_X=r-1-\dim(X^ *)$ the \emph{dual defect} of $X$.
The  general fibre of 
$p_2: \cP_X\to X^ *$
isomorphically projects to $\PP^ r$ to a linear space of dimension $\epsilon_X$ (see \cite {GH}). Namely, the \emph{general tangent  hyperplane} to $X$, i.e. the hyperplane corresponding to the general point of  $X^ *$,  is tangent to $X$ along a not empty open subset of a linear space of dimension $\epsilon_X$, which is called the \emph{contact locus} of the tangent hyperplane. 
The intersection of $X$ with a general tangent hyperplane is called a \emph{general tangent hyperplane section} of $X$. 
For smooth dual defective 
varieties one has $\epsilon_X\equiv n$ modulo 2 by \emph{Landman's Parity Theorem} (see \cite{Ein}).

\begin{lemma}\label{dimdual} Let $X\subset \PP^ r$ be an irreducible,
 non--degenerate variety of dimension $n$. Then $\dim(X^ *)\geq r-n$.
If $n\geq 2$ and equality holds, then  $X\subset\p^r$ 
is a singular scroll over a curve.
\end{lemma}

\begin{proof} 
Since $X\subset\p^r$ is non-degenerate, we have $\epsilon_X=r-1-\dim(X^*)\leq n-1$,
thus $\dim(X^*)\geq r-n$, with equality if and only if  $\epsilon_X=n-1$. Hence for $n\geq 2$  the variety
is singular by  Landman's parity theorem and the general contact locus is a linear space of dimension $n-1$. 
A classical result of Del Pezzo, \cite{dP} and \cite[Corollary 15]{Rog}, implies that $X\subset\p^r$ is a scroll over a curve.
\end{proof}

\begin{corollary}\label{tanred} Let $X\subset \PP^ r$ be an irreducible, non--degenerate
 variety of dimension $n\geq 2$. If the general tangent hyperplane section of $X$ 
is reducible, then either $X$ is a scroll over a curve, or it is a cone over the Veronese surface
 of degree 4 in $\PP^ 5$ or over one of its  projections.
\end{corollary}

\begin{proof} The assertion follows from Lemma \ref {dimdual} and  \cite [Corollary 13]{Rog}.
\end{proof}

\subsection{Cremona transformations}\label{ssec:cremona}

A \emph{Cremona transformation} is a birational map $\phi: \p^r\map \p^r$. We recall some basics about them.
We assume $r>1$.  

For a scheme $B\subset\p^r$, consider $\pi:\Bl_B(\p^r)\to\p^r$, the blow--up of $\p^r$ along $B$, and let $E$ be the exceptional divisor, which is a Cartier divisor. 
The scheme $\Bl_B(\p^r)$ is irreducible and reduced, hence a complex projective variety.

Let $\phi_1:\p^r\dasharrow \p^r$ be a Cremona
transformation  of \emph{type} $(d_1,d_2)$
with inverse $\phi_2:\p^r\dasharrow \p^r$, i.e. $\phi_i$ is defined by a fixed component free \emph{homaloidal}
linear system $\mathcal L_i$ of dimension $r$ of degree $d_i$ hypersurfaces.
 If $r=2$, then $d_1=d_2$.  We denote by $B_i$ the base locus scheme
of $\mathcal L_i$, i.e. the indeterminacy locus of $\phi_i$. 
One has $d_i=1$ if and only if $d_{3-i}=1$ in which case $\phi_1$ and $\phi_2$ are \emph{projective transformations}. We will assume from now on this is not the case.

The blow--ups $\pi_i:\Bl_{B_i}(\p^r)\to \p^r$ fit in the following 
diagram of birational maps
\[
\xymatrix{
&  \Bl_{B_1}(\p^r)=\Bl_{B_2}(\p^r)\ar[dl]_{\pi_1}\ar[dr]^{\pi_2}       &\subset \quad \p^r\times\p^r             \\
  \p^r  \ar@{-->}[rr]_{\phi_1} & &\p^r              }
\]
where $\pi_i$ are the restriction of the projections to the factors.
The equality $\Bl_{B_1}(\p^r)=\Bl_{B_2}(\p^r)$ follows from the fact that 
both schemes are irreducible, reduced, and both coincide with the closure of the graph of the maps $\phi_i$ in $\p^r\times\p^r$.

Let $E_i\subset \Bl_{B_i}(\p^r)$ be the exceptional divisors.
Let $H_i\in
\vert \pi_i^*(\mathcal O_{\p^r}(1))\vert$. Consider $r+1$ global sections $f_0,\dots, f_r$ of $\cO_{\p^r}(d_1)$ generating $\mathcal L_1$ and thus defining $B_1$. They
lift to sections $\pi_1^*(f_j)$ vanishing on $E_1$ and 
generating an invertible subsheaf $\mathcal L$ of $\pi_1^*(\cO_{\p^r}(d_1))$, precisely $\mathcal L\cong \pi_1^*(\cO_{\p^r}(d_1))\otimes \cO_{\Bl_{B_i}(\p^r)}(-E_1)$.
These sections of $\mathcal L$ define a morphism  $\tilde\phi_1:\Bl_{B_i}(\p^r)\to \p^r$, whose restriction to $\pi_1^{-1}(\p^r\setminus B_1)\cong \p^r\setminus B_1$ coincides with $\phi_1$. Then $\tilde\phi_1=\pi_2$, i.e. the diagram commutes, and 
\[
\pi_2^*(\cO_{\p^r}(1))\cong \mathcal L\cong  \pi_1^*(\cO_{\p^r}(d_1))\otimes\cO_{\Bl_{B_i}(\p^r)} (-E_1).
\]

Using the additive notation for Cartier divisors, denoting by $\equiv$ the linear equivalence and exchanging the roles of $\phi_1$ with $\phi_2$, we find
\[
H_2\equiv d_1H_1-E_1,\;\;\; H_1\equiv d_2H_2-E_2.
\]
>From this we deduce  
\begin{equation}\label{secondrel}
E_1\equiv (d_1d_2-1)H_2-d_1E_2,\;\;\; E_2\equiv (d_1d_2-1)H_1-d_2E_1.
\end{equation}

  Let $C_i$ {be the scheme theoretic image via $\pi_{i}$ of $E_{3-i}$, which is also  the
{\it ramification scheme} of $\phi_i$, $1\le i\le 2$}. The  {scheme $C_1$} has the same support
as the \emph{jacobian scheme} of $f_0,\dots, f_r$, defined by $\det (\frac {\partial (f_0,\dots, f_r)}{\partial (x_0,\ldots,x_r)})=0$.
Then  $C_i\subset\p^r$ is a hypersurface of degree $d_1d_2-1$ with equation $g_i=0$.
{If $I_{B_i}$ is the saturated ideal of $B_i\subset\p^r$,
\eqref{secondrel} yields  that $C_i$  \emph {has points of multiplicity $d_{3-i}$ along} $B_i$:
by this we mean that
the partial derivatives of $g_i$ of order  $d_{3-i}-1$ belong to $I_{B_i}$.}
Moreover $h^0(\p^r,\I_{B_i}(d_i))=r+1$,
$h^0(\p^r, \I_{B_i}(d_i-1))=0$, $h^0(\p^r,\I_{B_i}(1))=0$. In particular  $B_i$ is not a
complete intersection.

Let us restrict to the case of a \emph{quadro--quadric} Cremona transformation, i.e. $(d_1,d_2)=(2,2)$
(if either one of $d_1,d_2$ is $2$, the transformation is called \emph{quadratic}). 
Then the cubic hypersurface $C_i$ has double points along $B_i$.
Every line in $\Sec_b(B_i)$ is contracted by $\pi_i$ to a point, hence it is
contained in $C_i$. We claim that $C_i=\Sec_b(B_i)$, hence $C_i$ is the unique cubic hypersurface with double points along $B_i$.
Indeed, let $L\subset\p^r$ be a line which does not meet $B_1$ and let $\phi_1(L)=C$. Then $C$ is
a conic and the plane $\Pi_2=\langle C\rangle\subset\p^r$ cuts $B_2$ in a 
length  3 subscheme which is the intersection of $C$ with $B_2$. Then $\phi_2(\Pi_2)=\Pi_1\subset\p^r$ is a plane containing $L$. The  claim follows from  the well known classification of plane quadratic transformations.

\section{OADP and  Bronowski varieties} 

\subsection {Definitions and basic properties} Let $X\subset \PP^{2n+1}$ be an irreducible, non--degenerate, variety of dimension $n$ which is not secant defective, i.e. $\Sec(X)=\PP^ {2n+1}$.
One says that $X$ is a \emph{OADP variety} (OADP stands for \emph{one apparent double point})
if the morphism
$p_X: \SSS(X)\to \PP^ {2n+1}$ is birational, i.e.  if 
there is a unique secant line to $X$ passing through a general
point of $\PP^{2n+1}$.  We will say that $X$ is a \emph{Bronowski variety} if
its \emph{general tangential projection} is birational, i.e. if $x\in X$ is a general point, the projection
$\tau_x: X\dasharrow \PP^ n$
of $X$ from $T_{X,x}$ is birational. By Terracini's lemma, a Bronowski variety is not secant defective.

An irreducible, non--degenerate, variety $X\subset \PP^r$ is called 
\emph{linearly normal} if either one of the following equivalent
properties holds:
\begin{itemize}
\item [(LN1)] there is no variety $X'\subset \PP^{r'} $ with $r'>r$ and a
linear projection $\phi: \PP^ {r'}\dasharrow \PP^ {r}$ inducing an isomorphism $\phi: X'\to X$;
\item [(LN2)] the linear system $\vert \cO_{X}(1)\vert$ has dimension $r$.
\end{itemize}

Smooth OADP varieties have been studied in \cite {OADP}. 
As pointed out in Remark 1.2 there, several properties of smooth
OADP varieties are shared by singular OADP varieties. For instance:
\begin {itemize}
\item [(P1)] OADP varieties are {linearly normal}
(the proof of \cite [Proposition 1.1] {OADP} applies with no change);
\item [(P2)] OADP varieties are Bronowski varieties (the argument of \cite [\S 3] {OADP} applies with no change). In particular they are rational;
\item [(P3)] a secant line $L$ to a OADP variety $X$ is a \emph{focal line} if and only if for
some $z\in L-L\cap X$ there is
some other secant line, different from $L$, containing $z$. This happens
if and only if for  every $z\in L$, there
are infinitely many secant lines containing $z$. The Zariski closure $F(X)$ of the union of all
focal lines is the \emph{focal locus} of $X$ and $\dim(F(X))\le 2n-1$. For the general theory of foci, see
\cite {cs}; for the specific application to secant lines to OADP varieties, see  \cite[\S 1] {OADP}.
\end{itemize}

The following assertion goes back to Bronowski in \cite {Br}, who did not provide a satisfactory proof.

\begin{conjecture}[Bronowski's conjecture]\label{conj:bron} Any Bronowski variety is OADP.
\end{conjecture}

\subsection{OADP curves}  The classification of OADP curves follows right away from property (P3):

\begin{proposition}\label{prop:curves}
The only OADP curves are rational normal cubics.
\end{proposition}
\begin{proof} Let $X$ be a OADP curve in $\PP^3$ and let $d$ be its degree.
Let $\Gamma$ be the general plane section of $X$, which is in linear general position. 
Assume $d\ge 4$ and let
$p_i\in \Gamma$ be distinct points, with $1\le i\le 4$. The lines joining $p_1,p_2$ and
$p_3,p_4$ intersect,
hence the general secant line should be a focal line, which is impossible
because $\dim(F(X))=1$. 
\end{proof}

Another proof of the same result can be obtained by showing that a Bronowski curve
is a rational normal cubic, thus proving Bronowski's conjecture in dimension one. 
We do not dwell on this here.
 
\subsection{Geometric linear normality}\label{ssec:linorm}
 An irreducible, non--degenerate, variety $X\subset \PP^n$ is called 
\emph{geometrically linearly normal} if either one of the following equivalent
properties holds:
\begin{itemize}
\item [(GLN1)] there is no variety $X'\subset \PP^{r'} $ with $r'>r$ and a
linear projection $\phi: \PP^ {r'}\dasharrow \PP^ {r}$ inducing a birational morphism $\phi: X'\to X$;

\item [(GLN2)] if $p: X'\to X$ is a desingularisation, then the linear system $p^*(\vert \cO_{\PP^r}(1)\vert)$ is complete;

\item [(GLN3)] if $X'$ is smooth,  $\phi: X'\dasharrow X$ is a birational map
and $\cL$ is the strict transform  on $X'$ via $\phi$ of $\vert \cO_{\PP^r}(1)\vert$, 
then $\cL$ is \emph{relatively complete}, i.e. 
it is the complete linear system on $X'$ with
the same class in ${\rm Pic}(X')$ and with the same base point scheme as $\cL$. 
\end{itemize}
The reader may check the equivalence of the three above properties. Geometric linear normality
implies linear normality, but the converse does not hold, although it does if the variety is {normal, see the
argument in Proposition \ref{prop:normal} below}. 

A more precise information than property (P1) is available.

\begin{proposition}\label{prop:normal} Let $X\subset \PP^{2n+1}$ be a
OADP variety of dimension $n$. If $X$ is normal, then it is geometrically 
linearly normal.
\end{proposition}
\begin{proof} Suppose the assertion {were} not true. Then
there is a variety $X'\subset \PP^{r} $, with $r>2n+1$, and a
linear projection $\phi: \PP^ {r}\dasharrow \PP^ {2n+1}$ 
inducing a finite birational morphism $\phi: X'\to X$. Since $X$ is
normal, $\phi$ is an isomorphism by Zariski's Main Theorem, contradicting (P1).\end{proof}

\subsection{Hyperplane sections} By imitating the proof  of \cite [Proposition 4.6] {OADP}, one has:

\begin{proposition} \label{hyperlane} Let $X\subset \PP^{2n+1}$ be 
a Bronowski variety (e.g. a OADP variety) of dimension $n\geq 3$, which is not a scroll over a curve. Then:
\begin{itemize}
\item[(i)]  the general tangent hyperplane section is irreducible and rational;
\item[(ii)] if $Y$ is a desingularization of the general hyperplane section of $X$, one has
$\kappa(Y)=-\infty$.
\end{itemize}
\end{proposition}

\begin{proof} The irreducibility of the general tangent hyperplane section $D$ of $X$ follows from Corollary \ref {tanred}. Indeed a cone over the Veronese surface of degree 4 in $\PP^ 5$ is never  a Bronowski variety. The rationality of $D$ follows from the fact that the general tangential projection of $X$ is birational. This proves (i).

Consider a desingularisation $p: X'\to X$, let $Y$ be the proper transform on $X'$ of  $D$ and let $\cL=p^ *(\vert \cO_X(1)\vert)$.  Consider the subscheme $\cZ$
 of $\cL$ of dimension equal to $\dim(X^ *)=2n-\epsilon_X$,
  whose general point $Z$ is the strict transform on $X'$ of the general tangent hyperplane section of $X$. Then $Z$ is singular along the strict transform $P$ on $X'$ of the corresponding contact locus. 
  The tangent space $T_{\cZ,Z}\subset \cL$ has codimension $\epsilon_X+1$. By standard deformation theory (see \cite {WDV}, \S 2), $T_{\cZ,Z}$ is contained in the linear system of divisors in $\cL$ containing $P$. This shows that $Z$ has double points along $P$ with no other proper or infinitely near singularity. In particular, $Z$ has canonical singularities along $P$. The proof now proceeds as the one  of \cite [Proposition 4.6] {OADP}.  \end{proof}
  
\subsection{The fundamental hypersurface}\label{ssec:fund}  Let $X\subset \PP^{2n+1}$ be 
a Bronowski variety of dimension $n$, e.g. $X$ is  a OADP variety. Let $x\in X$ be a general point and let $\pi: \tilde X=\Bl_x(X)\to X$ be the blow--up
of $X$ at $x$, with $E\cong \PP^{n-1}$ the exceptional divisor. Set $\tildetau_x=\tau_x\circ \pi: \tilde X\dasharrow \PP^n$. The map $\tildetau_{x\vert E}$ is defined on $E$ by the linear system of quadrics given by
the \emph{second fundamental form} ${\rm II}_{X,x}$ of $X$ at $x$. In order to ease notation, we may drop the indices $X$ and/or $x$, if there is no danger of confusion. 

By \cite [(5.7)] {GH}, $\tildetau_{\vert E}$ 
birationally maps $E$ to a hypersurface $V\subset \PP^{n}$, which we call
the \emph{fundamental hypersurface} of $X$. We denote by $\delta$ its degree. 
Since $V$ is a birational projection to $\PP^n$ of the $2$--Veronese variety of $\PP^{n-1}$, one has
$1\le \delta\le 2^{n-1}$.  To ease notation we denote the inverse of $\tau$ by $\sigma$, $\tildetau_{\vert E}$ by $\bar \tau$ and by $\bar \sigma$ its inverse.

\begin{remark}\label{rem:planecase}{\rm If $X$ presents the \emph{hyperplane case}, i.e. if $V$ is a hyperplane, there are two possibilities:
\begin{itemize}
\item [(H1)] $\II_x$ has a base hyperplane and $\bar \tau: E\map V$ is a projective transformation;
by \cite [(3.21)] {GH}, if $n\ge 3$ then $X$ is a scroll over a curve;  in \S \ref {ssec:H1} we will see that the same happens also for
$n=2$;
\item [(H2)] the general quadric in $\II_x$ is irreducible and $\bar \tau: E\map V$
is a quadratic Cremona transformation.
\end{itemize} 
}
\end{remark}

 The map $\sigma$ is defined by a
linear system $\cX$
of dimension $2n+1$ of hypersurfaces of degree $d\le \deg(X)$ in $\p^n$, which is the image via $\tau$ of
the linear system of hyperplane sections of $X$. The system $\cX$ is relatively complete if and only if $X$ is geometrically linearly normal. This is the case if $X$ is a normal OADP. 
Since $\sigma$ contracts $V$ to $x\in X$ we have:

\begin{lemma}\label{lem:1}
$\cX$ has no movable intersection with $V$. 
\end{lemma}

Denote by $\cX'$ the relatively complete linear system of dimension $2n$ of hypersurfaces of degree $d-\delta$  in $\PP^n$
formed by all hypersurfaces $F$ such that $F+V\in \cX$. The image of the map $\phi_{\cX'}: \PP^n\map 
\PP^{2n}$ is the same as the one of $X$ via the \emph{internal projection} $\pi_x: X\map X'\subset \PP^{2n}$
from $x$. Hence:

\begin{lemma}\label{lem:2}
 The restriction of $\phi_{\cX'}$ to $V$ is $\bar \sigma$.
 \end{lemma}

Denote by $\cX''$ the linear system of hypersurfaces of degree $d-2\delta$ in $\PP^n$
formed by all hypersurfaces $F$ such that $F+2V\in \cX$. The movable part of $\cX''$ is 
the linear system of all hyperplanes of $\PP^n$, corresponding, via $\tildetau$ to the
tangent hyperplane sections of $X$ at $x$. In conclusion:

\begin{lemma}\label{lem:3}
The fixed part, if any, of $\cX''$ consists of  the 
exceptional divisors corresponding to the indeterminacy locus of $\tildetau$ off $x$.
Moreover $d\ge 2\delta+1$ with equality if and only if $\cX''$ has no fixed part.
\end{lemma}

\begin{remark}\label{rem:cremona} 
{\rm
If $X$ is not a scroll and $p\in \PP^n$ is a general point,
the linear system $\cX_p$ of all hypersurfaces in $\cX$ which are singular at $p$,
is a homaloidal system, i.e. it defines a Cremona transformation of $\PP^n$.
}
\end{remark}

\section{Verra varieties}\label{sec:verra}

We present now a construction of OADP varieties following an idea of A. Verra.

Let  $Y\subset {\PP}^{2n+1}$ be a degenerate OADP variety of dimension $m<n$,
which spans a linear space $V$ of dimension $2m+1$. 
Take a linear space $W\subset \PP^ {2n+1}$ 
of dimension $2(n-m)-1$ such that $V\cap W=\emptyset$.
Let $C_W(Y)$ be the \emph{cone}
over $Y$ with vertex $W$. Let $X\subset C_W(Y)$ be an irreducible, non--degenerate, not secant defective variety of dimension $n$, which intersects the general ruling $\Pi\cong \PP^ {2(n-m)}$ of $C_W(Y)$ 
along a linear subspace $P$ of dimension $n-m$.
This implies that:
\begin{itemize}
\item [(A1)]  the projection
$p: \PP^ {2n+1}\dasharrow V$ of $\PP^ {2n+1}$ 
from $W$ to $V$ restricts to $X$ to a dominant map $p_{\vert X}: X\map Y$;

\item [(A2)] if $P_i$, $1\le i\le 2$, are the closures of two general fibers of $p_{\vert X}$, then
$P_1\cap P_2=\emptyset$.
\end{itemize}

Indeed, (A1) is clear, and (A2) follows, via Terracini's Lemma, from the fact that $X$ is not secant defective. 
We will call these varieties  \emph{Verra varieties}.

\begin{example} {\rm Let $Y$ be as above. 
Consider a rational map $f: Y\map \G(n-m-1, W)$ such that,
if $x,y\in Y$ are general points, then $f(x), f(y)$ are skew 
linear subspaces of $W$. Set
\[
X:=\overline { \cup_{x\in Y_0}\langle x,f(x)\rangle }
\]
where $Y_0\subset Y$ is the open subset where $f$ is defined. Then $X$ is a Verra
variety.}
 \end{example}

\begin{proposition} [A. Verra] \label{prop:verra}  Verra varieties are OADP.
\end{proposition}

\begin{proof} Let $X$ be a Verra variety.  Let $x\in \PP^ {2n+1}$ be a general point, so that $y=p(x)$
is a general point of $V$. A secant line to $X$ through $x$ is a general secant line to
$X$ and projects to a general secant line to $Y$ passing through $y$. Since $Y$ is OADP,
there is only one such secant line $L$ intersecting $Y$ at two points $p_1,p_2$. 
Hence all secant lines to $X$ 
through $x$ lie in the linear space $Z=\langle L\cup W\rangle$ of dimension $2(n-m)+1$, 
which intersects $X$ along the two linear spaces $P_i\subseteq \langle p_i, W\rangle$, 
$1\le i\le 2$, of dimension $n-m$,  whose union spans
$Z$. The assertion follows, since there is only one secant line to $P_1\cup P_2$ passing through
$x\in Z$. \end{proof}

\begin{remarks}\label{rem:verra} {\rm (i) Starting the construction with $Y$ secant defective,
one finds $X$ secant defective.

(ii) Verra varieties are in general singular. For instance, the only smooth Verra surface is 
the rational normal scroll $S(1,3)$ in $\PP^5$ (for the notation regarding scrolls, see \cite {OADP}). It is a nice question, on which we will not dwell here,  to classify all smooth Verra varieties. Even more interesting is the question of which Verra varieties are smoothable. }
\end{remarks}

\section{Singular OADP surfaces}\label{sec:singularOADP}

{OADP surfaces with finitely many singular points} have been classified 
by Severi in \cite {Se} (but his proof was incomplete) and then by Russo \cite {Ru} 
(see also \cite {OADP, violo}).  They are only rational normal scrolls and del Pezzo surfaces of degree 5.  In this section we will prove the following classification which has been announced also by F. Zak. 

\begin{theorem}\label{thm:classif} Let $X$ be a OADP surface. 
Then $X$ is either a smooth rational normal scroll, or a (weak) del Pezzo surface of degree 5 or a Verra surface.
\end{theorem}

\begin{proof} Assume first  $\dim(\Sing(X))=1$. 
Let $Z$ be an irreducible curve contained in $\Sing(X)$. It is not contained in $T_{X,x}$ for
$x\in X$ general, otherwise $X$ would be a cone, hence defective.

\begin{claim}\label{claim1} The general tangential projection $\tau_x$ maps $Z$ to a point.
\end{claim}
\begin{proof} Let $\pi: X'\to X$ be a minimal resolution of the singularities of $X$ and of $\tau_x$ so that $\tau'_x=\tau_x\circ \pi: X'\to  \PP^2$ is a birational morphism. Let $Z'$ be the proper transform of $Z$ via $\pi$ and let $z\in Z$ be a general point. 
The pull--back scheme $z'$ of $z$ via $\pi$ is finite of length at least 2, and it is mapped by $\tau_x'$ to the point $y=\tau_x(z)$. By Zariski's Main Theorem, there is an irreducible curve $C_z$ on $X'$ containing $z'$ contracted to $y$ by  $\tau_x'$. The curve $C_z$ does not depend on $z$, otherwise $\tau'_x$ would not be birational. Thus $C_z=C$, with $C$ a fixed curve. On the other hand, $C$ contains $z'$, which describes the whole $Z'$ as $z$ moves on $Z$. Thus $C=Z'$. This proves the Claim.
\end{proof}

\begin{claim}\label{claim2} $Z$ is a line which has non--empty intersection with the tangent planes to $X$ at smooth points.
\end{claim}

\begin{proof} Set $V=\langle Z\rangle$. By Claim \ref {claim1}, $T_{X,x}\cup V$ has dimension 3 for $x\in X$ general. Hence $\dim(V)\le 3$ and the equality does not hold, otherwise we would have $T_{X,x}\subset V$ for $x\in X$ general. If $\dim(V)=2$, then $T_{X,x}\cap V$ would be a line, hence two general tangent planes to $X$ would intersect and, by Terracini's lemma,  $X$ would be secant defective, a contradiction.  This proves the Claim.
\end{proof} 

Let $W\subset \PP^5$ 
be a linear space of dimension 3 such 
that $V\cap W=\emptyset$ and let $f: \PP^5\dasharrow W$ 
be the projection from $V$. Since the general tangent plane to $X$ 
intersects $V$ at a point, the image of $X$ via $f$ is a curve $Y$ which spans $W$. 
Let $p\in \PP^5$ be a general point and set $q=f(p)$ which is a general point of $W$.
 Let $L$ be a secant line to $Y$ through $q$, meeting $Y$ at two points $p_1,p_2$. Let $F_i$ be the fibre of $p_i$ via $f$, which is a curve in the plane $P_i=\langle p_i,V\rangle$, for $1\le i\le 2$. Look at $T=\langle L\cup V\rangle$, which has dimension 3, and contains $p, P_1, P_2$. If $F_1,F_2$ have degree larger than 1, there would be more than one secant line to $F_1\cup F_2$ passing through $p$, which is not possible. Thus the general fibre of $f_{\vert X}: X\dasharrow Y$ is a line. By a similar argument, one sees that $Y$ has to be a OADP curve, hence a rational normal cubic. In conclusion, $X$ is a Verra surface. 

Assume next $\dim(\Sing(X))=0$, which yields $T_{X,x}\cap \Sing(X)=\emptyset$ for $x\in X$ general, otherwise $X$ would be a cone, hence defective.
Since the general tangential projection $\tau_x: X\dasharrow \PP^2$ is birational, the argument  of \cite
[Proposition 4.7] {OADP} applies to show that
\begin{itemize}
\item [(i)] either  $X$ is a rational scroll, or
\item [(ii)] the general hyperplane section $C$ of $X$ is irreducible of geometric genus $1$. 
\end{itemize}
 
Let $\pi: X'\to X$ and $\tau'_x: X'\to \p^2$ be as above. Any curve contracted by $\pi$ is also contracted by $\tau'_x$, hence it is rational. This implies
$h^1(X,\cO_X)=h^1(X',\cO_{X'})=0$. 

In case (i) the set theoretic intersection of $X$ with $T_{X,x}$ 
consists only of the ruling $L$ of $X$ passing through $x$, i.e. 
there are no isolated intersection points of  $X$ with $T_{X,x}$ off $L$ 
(see  \cite  [Proposition 6.3] {OADP}, which holds, even if $X$ is singular,
 provided $h^1(X,\cO_X)=0$ and the general curve $C\in \vert \mathcal M\vert$ 
is contained in $X\setminus \Sing(X)$). Furthermore, the general tangent curve section 
with a hyperplane through $T_{X,x}$ consists of $L$ plus an irreducible curve $C$  which is smooth at $x$ and its tangent line is not fixed. If $d$ is the degree of $X$, then $C$ has degree $d-1$ and $\tau_x(C)$ has degree $d-3$. On the other hand $\tau_x(C)$ is a line. Thus $d=4$ and $X$ is a rational normal scroll. 

In case (ii) the same argument proves that $d=5$ and $X$ is a (weak) del Pezzo surface.\end{proof}

\section{The hyperplane case}\label{sec:hypercase}

In  this section we take up the classification of Bronowski (in particular OADP) varieties, presenting {\it the hyperplane case}, with the two subcases $(H1)$ and $(H2)$ indicated in Remark \ref {rem:planecase}. 

\subsection{The cubic invariant}\label{ssec:generalities}

Notation as in \S\ref {ssec:fund}. 

\begin{proposition}\label {prop:triplept} Let $X\subset \p^{2n+1}$ be a Bronowski variety and let $x\in X$ be a general point. Then:
\begin {itemize}
\item [(i)] $X$ presents the hyperplane case if and only if there is a hyperplane section $H_x$ of $X$ with a point of multiplicity 
at least 3 at $x$;
\item [(ii)] if $X$ presents the hyperplane case, then the hyperplane section $H_x$ is unique and has multiplicity exactly 3 at $x$.
\end{itemize} 
\end{proposition}
\begin{proof} The tangent hyperplane sections of $X$ at $x$ are mapped by $\tau$ to the hypersurfaces of the linear system $\cX''$, whose movable part is the linear system of all hyperplanes of $\p^n$. There is a tangent hyperplane section
with a point of multiplicity 3 or more if and only if the fundamental hypersurface is a hyperplane. This proves (i) and the uniqueness assertion in (ii).

Let $m\ge 3$ be the multiplicity of $H_x$ at $x$. Consider the positive dimensional 
subvariety $\mathcal T$ of $X^*$ whose general point is $H_x$. 
Take a point $H\in T_{\mathcal T,H_x}$ different from $H_x$. It is
a hyperplane section $H$ of $X$ with a point of multiplicity at 
least $m-1$ at $x$ (see \cite {WDV}, \S 2).  By the uniqueness of $H_x$, $m=3$ follows.
\end{proof}

We denote by $\tilde H_x$ the proper transform of $H_x$ on $\tilde X$. Its intersection with the
exceptional divisor $E$, is a cubic hypersurface $C_x$ corresponding to  
the tangent cone to $H_x$ at $x$. To ease notation, we may drop the index $x$ if there 
is no danger of confusion.
We call $C$ the \emph{cubic invariant} of $X$.

\begin{corollary}\label{cor:cubicinv} Same hypotheses as in Proposition \ref {prop:triplept}. If $H_x$
has an isolated triple point at $x$ (in particular if the cubic invariant is not a cone), then we are in case $(H2)$ and the quadratic Cremona transformation
$\bar \tau: E\map V$ is \emph{homaloidal} {\rm(}see \cite {crs}\rm{)}. 
\end{corollary}

\begin{proof} In case $(H1)$, $X$ is a scroll over a curve, and $H_x$ is triple along the whole ruling through $x$. So we are in case $(H2)$.  Then the variety $\mathcal T$ in the proof of Proposition \ref {prop:triplept}
has dimension $n$ and $T_{\mathcal T,H_x}$ is spanned by $H_x$ and by independent hyperplane
 sections $H_i$, $1\le i\le n$, with a double point at $x$,
 which span the linear system of all hyperplane sections tangent to $X$ at $x$. Locally around $x$, they are given by the first order derivatives of the equation of $H_x$ (see again \cite[ \S 2]  {WDV}). This implies that the quadrics in ${\rm II}_{X,x}$ are spanned by the derivatives of the polynomial defining $C$. \end{proof}

\begin{remark}\label{rem:pres} The same argument of the proof of Corollary \ref {cor:cubicinv}
shows that, in any case, the first polars of the cubic invariant belong to the second fundamental form,
see also \cite[Theorem 5.2]{PR} and \S \ref{parametric} below.
\end{remark}

\subsection{The $(H1)$ case} \label{ssec:H1}

\begin{lemma}\label{lemma:scrolls} Let $X\subset \PP^{r}$ be an irreducible,
non--degenerate, normal variety of dimension $n\ge 2$, which is a scroll over a curve. Then either $X$
is smooth or it is a cone.
\end{lemma}
\begin{proof} Let $x\in X$ be a singular point. We claim that
$X$ is a cone with vertex at $x$. By taking hyperplane sections through $x$,
it suffices to prove the assertion if $X$ is a surface, which follows from \cite [Claim 4.4] {CLM}. 
\end{proof} 

Using this, by taking into account  Theorem \ref {thm:classif}, Remark \ref {rem:planecase},  \cite [Example 2.1] {OADP} and by remarking that  for surfaces presenting the hyperplane case  only $(H1)$ occurs, we deduce the following result.

\begin{proposition}\label{prop:scrolls} Let $X\subset \PP^{2n+1}$ be a normal Bronowski variety
of dimension $n$. 

\begin{itemize}

\item[(i)] If $X$ is a scroll over a curve, then it is a smooth rational normal scroll.

\item[(ii)] $(H1)$ holds for $X$
if and only if $X$ is a smooth rational normal scroll.
\item[(iii)] If $n=2$ then $X$ presents the hyperplane case if and only if it
is a smooth rational normal scroll.
\end{itemize}
\end{proposition}

\subsection{The $(H2)$ case}\label {ssec:Q1case}
In \cite{PR} one studies irreducible, non--degenerate projective varieties $X\subset\p^{2n+1}$ of dimension $n$
which are \emph{3--covered} by twisted cubics, i.e. given three general points of $X$ there is a twisted cubic
contained in $X$ passing through them. It has been proved there that any 
such variety is a Bronowski, and actually a OADP variety {presenting} the hyperplane case. It was also proved there that for these varieties $(H1)$ holds if and only if they are rational normal scrolls.
Next we somehow complete the picture by proving that the class of normal, geometrically linearly normal, Bronowski varieties of dimension $n$ 
presenting  case $(H2)$ coincides with the one of {normal}
irreducible  varieties $X\subset\p^{2n+1}$ which are 3--covered by twisted cubics, and are different from rational normal scrolls. In particular Bronowski's conjecture \ref {conj:bron} holds for {the previous class of} Bronowski varieties
presenting  case $(H2)$

\begin {proposition}\label{prop:22} 
Let $X\subset \PP^{2n+1}$ be a Bronowski variety presenting case $(H2)$.
Then $\bar\tau:E\map V$ is a quadro--quadric Cremona transformation and the cubic invariant is the ramification scheme of $\bar \tau$.
\end{proposition}

\begin{proof} Let $k$ be the degree of the inverse $\bar \sigma$ of $\bar \tau$.
The element of the linear system $\cX''$ given by the hyperplane $V$
plus the fixed part (if any), corresponds to the divisor $\tilde H$
on $\tilde X$ which cuts out the cubic invariant $C$ on $E$. Since $V$ is
contracted by $\sigma$ to a point, $\tilde H$ is the image of the
indeterminacy locus of $\sigma$. Therefore the ramification scheme
of $\bar \tau$, of degree $2k-1$, is contained in $C$. 
Thus $2k-1\le 3$, hence $k=2$ and the assertion follows.
\end{proof}

Let $Z\subset E$ be the indeterminacy locus scheme of $\bar\tau$ and let $Z'\subset V$ be the same for $\bar\sigma$. As we saw in \S \ref {ssec:cremona}, we can consider the cubic ramification schemes $C\subset E$ (i.e. the cubic invariant by Proposition \ref {prop:22}) and $C'\subset V$. The hypersurface  $C$ [resp. $C'$]
has double points along $Z$ [resp. along $Z'$]. Moreover, $C=\Sec_b(Z)$ [resp. $C'=\Sec_b(Z')$] is the unique cubic hypersurface in $E$ [resp. in $V$] with double points along $Z$ [resp. along $Z'$]. In particular,  
the linear system of cubic hypersurfaces in $\p^n$ having double points along $Z'$ has dimension  $2n+1$.  

We may assume that $V\subset\p^n$ has equation $x_0=0$.
By the analysis in \S \ref {ssec:fund}, the vector space of forms of degree $d$ corresponding to 
the linear system $\cX$ can be generated by a polynomial $g$ not divisible by $x_0$ and  by $2n+1$ polynomials $g_1,\ldots, g_{2n+1}$, each divisible by $x_0$ and $n+1$ of them actually divisible by $x_0^2$. The latter polynomials span a vector space corresponding to the $n$--dimensional linear system of hyperplane section of $X$ tangent to $X$ at $x$.

 \begin{theorem}\label{th:hyperplanar} Let $X\subset \PP^{2n+1}$ be a normal, geometrically linearly normal, 
 Bronowski variety presenting case $(H2)$. Then the inverse $\sigma:\p^n\map X\subset\p^{2n+1}$ of the general tangential projection of $X$ is defined
 by the relatively complete linear system of cubic hypersurfaces  having double points along $Z'$.
 Equivalently $X\subset\p^{2n+1}$ is  a  variety 3-covered by twisted cubics and therefore it is a OADP variety.
 \end{theorem}
 
 \begin{proof} Let  $d\geq 3$ be  the degree
 of the linear system of hypersurfaces defining $\sigma$.
By Lemma \ref {lem:2}, the map $\phi_{\cX'}$ restricted to $V$ coincides with $\bar\sigma:V\map E$.

Consider the  codimension 2 complete intersection subscheme $B$ of degree $d\geq 3$ in $\p^n$ 
with equations $x_0=g=0$,
contained in the indeterminacy locus scheme of $\sigma$. This 
can be considered as a non--zero divisor $F$ of $V$ defined by the equation $\bar g=0$, where $\bar g$ is $g$ modulo $x_0$. 
Since $\cX'_{\vert V}$ induces the quadro--quadric birational map $\bar \sigma$, the fixed component of this linear system is
a divisor $F'\subset V$ of degree $d-3$. 

\begin{claim}\label {cl:1} 
 One has $F=F'+C'$ as subschemes of $V$. 
 \end{claim}
 
 \begin{proof}[Proof of the Claim] 
 Let  $Y_1=\Bl_B(\p^n)$ with exceptional divisor $B_1$, and let $W_1$ be the proper transform of $V$ on $Y_1$. Then  $W_1\cong V$ and $Z'$ can be considered as subscheme of $W_1$. Let
$Y_2=\Bl_{Z'}(Y_1)$ and take a minimal desingularisation $Y$ of $Y_2$ (this is considered by technical reasons and, as a first approximation, the reader may pretend that $Y=Y_2$). Let 
$D=\sum_{i=1}^h D_i$ ($D_i$  irreducible) be the exceptional divisor of $Y\to Y_2$
and let $W$ and $\bar B_1$ be the proper transforms of $W_1$ and $B_1$
on $Y$. Let $\tilde \sigma:Y\map X$ be the induced birational map.

 Let $H_Y$ be the pull--back of the
 hyperplane class of $\PP^n$ on $Y$ and let $H_W$ be its restriction to $W$. Let $D_W=\sum _{i=1}^h D_{W,i}$   and  $\bar F$ be the restrictions
 of $D$ and $\bar B_1$ to $W$. Note that $\bar F$ is the proper transform of $F$ on $W$.
 Let $\cX_Y$ and $\cX'_Y$ be the proper transforms of $\cX$ and $\cX'$ on $Y$.  
The linear system $\cX_Y$ is complete by the geometric linear normality assumption, so $\cX'_Y$ is complete too.
We may abuse notation and denote by $\cX_Y$ and $\cX'_Y$  the associated line bundles. 
 Then $\cX_Y$ is of the form
 $\cO_Y(dH_Y-\sum_{i=1}^h m_iD_i-\bar B_1)$, where the $m_i$'s are non--negative integers. 
 By Lemma \ref{lem:1} one has $\cX_{Y\vert W}\cong \cO_W$.
 This implies that $\bar F\equiv dH_W-\sum_{i=1}^h m_iD_{W,i}$.
 Moreover
 \[
 N_{W,Y}\cong \cO_W((-d+1)H_W+\sum_{i=1}^h (m_i-1)D_{W,i})
 \]
hence
\[
\cX'_{Y\vert W}\cong N_{W,Y}^*\cong \cO_W((d-1)H_W-\sum_{i=1}^h (m_i-1)D_{W,i}).\]

The linear system $\cX'_{Y\vert W}$ determines the restriction of 
$\tilde \sigma$ to $W$, which is a resolution of $\bar\sigma: V\map E$.
Hence $\cX'_{Y\vert W}$  is relatively complete and 
it is also determined by a linear system contained in $\vert \cO_W(2H_W-D_W)\vert$. 
Thus the proper transform $\bar F'$ 
of $F'$ belongs to $\vert \cO_W((d-3)H_W-\sum_{i=1}^h (m_i-2)D_{W,i})\vert$
and the sections of the bundle $\cX'_{Y\vert W}$ vanish on divisors of the 
form $\bar F'+Q$, where $Q\in \vert \cO_W(2H_W-D_W)\vert$.

The proper transform $\bar C'$ of $C'$ belongs to $\vert\cO_W(3H_W-2D_W)\vert$ so that $\bar F'+\bar C'\in \vert\cO_W(dH_W-\sum_{i=1}^h m_iD_{W,i}))\vert$. 
Also $\bar F$ sits
in the same linear system. If $F_i\subset V$ is defined by the equation 
$\frac {\partial \bar g}{\partial x_i}=0$, $1\le i\le n$, then the proper transform $\bar F_i$ of $F_i$ on $W$ sits in $\cX'_{Y\vert W}$. 
So we have relations of the form $\frac {\partial \bar g}{\partial x_i}=f'q_i$, where $q_i$ is 
of degree $2$. By Euler's theorem, $f'$ divides $\bar g$, i.e.
$\bar g=f'\cdot h$, where $h$ is a cubic form. Then the proper transform $\bar G$ of $G=V(h)$ on $W$ sits in
$\vert\cO_W(3H_W-2D_W)\vert$. By the discussion in \S \ref {ssec:cremona}, this linear system  contains only
$\bar C'$. This proves the claim.
\end{proof}

\begin{claim}\label{cl:2} The map $\tilde \sigma_{\vert W}$ contracts the divisor $\bar F$. 
\end{claim}
\begin{proof} [Proof of the Claim] Otherwise $\tilde\sigma(\bar F)$
would be a hypersurface on $E$ and $\bar\tau$ would be defined on the general pointed
of each of its components.
 Then the general hyperplane section
of $X$, corresponding to the hypersurface $g=0$, would pass through $x$, a contradiction.\end{proof}

Finally $\bar F=\bar F'+\bar C'$ (by Claim \ref {cl:1}) is contained in the ramification locus 
of $\tilde \sigma_{\vert W}$ (Claim \ref {cl:2}), which is $\bar C'$ plus exceptional divisors of $W\to V$.
This yields $\bar F'=0$, hence $d=3$, and $m_i=2$ for all $i\in \{1,\ldots,h\}$.
\end{proof}

\subsection{Remarks and examples}\label {ssec:examples}

In  \cite {PR} several results are proved for irreducible, non--degenerate varieties $X\subset \p^ {2n+1}$ of dimension $n$ which are 3-covered by rational normal cubics.  In view of Theorem \ref {th:hyperplanar}, all of them can be applied to normal, geometrically linearly normal, Bronowski varieties $X\subset \PP^{2n+1}$ presenting case $(H2)$.  We briefly summarize here some of these applications.

\subsubsection{The Hilbert scheme of lines}

Let $X\subset \p^r$ be an irreducible projective variety and let $x\in X$. Let 
$\mathcal L_x$ be the Hilbert scheme of lines in $X$ passing through $x$, which
can be seen as a subscheme of the exceptional divisor of the blow--up of $X$ at $x$. 
The following is a consequence of  the results in  \cite  [\S 5] {PR}.

\begin{corollary} \label{cor:cone} Hypotheses and notation as in Theorem \ref {th:hyperplanar}. If $x\in X$ is a general point, then:
\begin{itemize}
\item [(i)] $\mathcal L_x$ is isomorphic to both $Z$ and $Z'$;
\item [(ii)] if $X$ is smooth, then $\mathcal L_x$ (hence $Z$ and $Z'$) is also smooth;
\item [(iii)] $Z'$ coincides with the singular locus scheme of the general element in the linear system $\cX$;
\item [(iv)]  $\bar \tau$ is \emph{involutory}, i.e. it coincides with its inverse up to projective transformations;
\item [(v)] the hypersurfaces $C$ and $C'$ are projectively equivalent.
\end{itemize}
\end{corollary}

\subsubsection{Parametric representation}\label{parametric}
Notation as in \S \ref {ssec:Q1case}. 
Assume that the point $p=[1,0,\ldots,1]\in \p^ {2n+1}$ corresponds via $\sigma$ to the general point $x$ of $X$. By Proposition \ref {prop:triplept} we may assume that $g=0$ has a triple point at $p$, hence $g$ does not depend  on $x_0$. Moreover $g_i=x_0f_i$, $1\le i\le n$, and we may assume the $f_i$'s also independent from $x_0$, defining $\bar \sigma$ on $V$. 
Equivalently the linear system generated by $f_1,\ldots,f_n$ corresponds to ${\rm II}_x$,
and $f_1=\ldots=f_n=0$ defines $Z'$. 
Finally we may assume $g_i=x_0^ 2x_{i-n-1}$, $n+1\le i\le 2n+1$. Thus,  in the same spirit of \cite [\S 5] {PR},  the following gives a parametric representation of (an open affine subset of) $X$
\[
{\bf t}= (t_1,\ldots,t_n) \in \mathbb A ^ n\to (t_1,\ldots,t_n,f_1({\bf t}),\ldots, f_n({\bf t}),g(\bf t))\in 
\mathbb A^ {2n+1}
\]
There is a relation between $g$ and $f_1,\ldots,f_n$. Set ${\bf x}=(x_1,\ldots, x_n)$ and  ${\bf f}({\bf x})=(f_1({\bf x}),\ldots, f_n({\bf x}))$. 
Part (iv) of Corollary  \ref {cor:cone} and the above analysis (see also, as usual, \cite[\S 5]{PR}), 
imply that there is a non--degenerate matrix $A$ of type $n\times n$  such that
\[
A\cdot {\bf f}({\bf {\bf f}({\bf x})})^ t=g({\bf x}){\bf x}^ t.
\]
This implies that, as pointed out in Remark \ref {rem:pres}, the derivatives of $g$ belong to the ideal $(f_1,\ldots, f_n)$. 

\subsubsection{Twisted cubics over Jordan algebras}\label{Jalg} 

The results in  \cite[\S 5]{PR} and in \cite{PR2} imply the equivalence of:

\begin{itemize}
\item [(i)] the classification, up to projective equivalence, of
irreducible varieties $X\subset \PP^{2n+1}$, $n\geq 3$, 3-covered by twisted cubics and different from
rational normal scrolls.
\item [(ii)] the classification, up to projective equivalence, of  involutorial  $(2,2)$ Cremona transformations of $\p^ n$, $n\ge 2$;
\item [(iii)] the classification, up to isomorphisms, of cubic, complex, unitary \emph{Jordan algebras}  $\mathbb J$ of dimension $n+1$.
\end{itemize}

Items (i), {as proved in   \cite[\S 5]{PR} and \cite{PR2}},  give rise to items (ii). These give rise to items (iii)
(see \emph{loc. cit}). Namely,  given an  involutorial $(2,2)$ Cremona transformation $\phi$ of $\p^ n$,
there is a cubic, complex, unitary Jordan algebra  $\mathbb J$, which is unique up to isomorphism,  such that $\phi$ coincides, up to a projective transformation, with the involution on $\p(\mathbb J)$ sending the class of an invertible element  $x\in \mathbb J$  to the class of its \emph{adjoint} $x^ \sharp$. Finally, again as  in   \cite[\S 5]{PR}, an  item $\mathbb J$ in (iii) gives rise to the item in (i) given by the  projective isomorphism class of the \emph{twisted cubic} $X_\mathbb J$ over $\mathbb J$, defined as 
the Zariski closure of the image of the map
\[
x\in \mathbb J\dasharrow [1, x, x^ \sharp, N(x)]\in \p(\mathbb C\oplus \mathbb J\oplus \mathbb J\oplus\mathbb C)
\]
where $N$ is the \emph{norm} of $\mathbb J$, i.e. the cubic form coinciding, up to a constant,  with the cubic invariant of $X_\mathbb J$.  In this setting, smooth items in (i) correspond to items in (ii) with smooth base locus and to semisimple Jordan algebras. 

Varieties of the form $X_\mathbb J$ are geometrically linearly normal (see \cite[Lemma 5.1]{PR}).
The analysis of the low dimensional examples supports evidence for conjecturing that every variety $X_\mathbb J$ should be also normal. Were this true,
Theorem \ref {th:hyperplanar} would imply that  (i)--(iii) are also equivalent to

\begin{itemize}
\item [(iv)] the classification, up to projective equivalence, of
normal, geometrically linearly normal Bronowski varieties $X\subset \PP^{2n+1}$, $n\ge 3$, presenting case $(H2)$.
\end{itemize}
\medskip

\subsubsection{Lagrangian Grassmannians}\label{ssec:laggrass}
The only simple, cubic, complex, unitary Jordan algebras are 
${\rm Sym}^3{\mathbb C}$, $M_3{\mathbb
C}$, ${\rm Alt}_6{\mathbb C}$ and $H_3{\mathbb O}$ (this follows from Jordan--von Neumann--Wigner classification theorem of simple, complex, unitary Jordan algebras, see \cite[p. 49] {mccrimmon}). 
Consider the twisted cubics over  these algebras. They coincide with the
{\it lagrangian
grassmannians} $\G^{\rm lag}_k(2,5)$ over the complexification of the four composition {real} algebras
$k=\mathbb R,\mathbb C,\mathbb H,\mathbb O$, in their
Pl\"ucker embedding (see \cite{Mu2} and \cite[\S 5]{PR} for  details).
Another description is as follows (see \cite{Zak2}):
\begin{itemize}
\item[(i)] $\G^{\rm lag}_{\mathbb R}(2,5)$ is the {\it symplectic grassmannian}
$\G^{\rm symp}(2,5)$ of
$2$--planes in $\Proj^5$ on which a non--degenerate \emph{null--polarity} is
trivial, i.e. they are isotropic with respect
to a non--degenerate antisymmetric form. This is a variety of dimension
$6$ embedded in $\p^{13}$ via its
Pl\"ucker embedding;
\item[(ii)]$\G^{\rm lag}_{\mathbb C}(2,5)$ is the \emph{grassmannian} $\G(2,5)$
of $2$--planes in $\Proj^5$ of
dimension $9$ sitting in $\Proj^{19}$ via the Pl\"ucker
embedding;
\item[(iii)]$\G^{\rm lag}_{\mathbb H} (2,5)$ is the \emph{spinor variety} $S_{5}$ 
of dimension $15$ in $\Proj^{31}$, {parametrizing linear $5$-dimensional subspaces of a smooth
quadric ${\mathbb Q}^{10}\subset\p^{11}$};
\item[(iv)] $\G^{lag}_{\mathbb O}(2,5)$ is the \emph{$E_7$--variety} of dimension
$27$ in $\Proj^{55}$.
\end{itemize}

If $n\geq 3$ and $\Q^{n-1}\subset\p^n$ is a smooth quadric, the
Segre embedding  of $\p^1\times \Q^{n-1}$ into $\p^{2n+1}$ is  the twisted cubic over  the unique
semi--simple, non simple, cubic complex Jordan algebra (see again \cite{Mu2},  \cite[\S 5]{PR} and 
Jordan--von Neumann--Wigner theorem  \cite[p. 49] {mccrimmon}). 

For lagrangian grassmannians (i)--(iv),
$\mathcal L_x$ is one of the four Severi varieties, i.e. 
the \emph{Veronese surface}
$V_{2,2}\subset \p^5$ in case (i); the \emph{Segre variety} ${\rm Seg}(2,2)\subset \p^8$ in case (ii);
the grassmannian $\mathbb G(1,5)\subset\p^{15}$ in case (iii); the \emph{Cartan variety} in   $\p^{26}$
in case (iv). 
By Corollary \ref {cor:cone} the lagrangian grassmannian of dimension  $n=6,9,15,27$,  
is the closure of the image of the birational map given by cubics in $\p^n$ singular along a degenerate Severi variety of dimension $m=2,4,8,16$. 

This viewpoint is valid also for
$\p^1\times \Q^{n-1}$, $n\geq 3$, which is closure of the image of $\p^n$ via 
the linear system of cubics singular along $p\cup \Q^{n-2}$, where $p\not\in \langle \Q^{n-2}\rangle$.

By the considerations in \S \ref {Jalg} we have:

\begin{corollary}\label{cor:smooth} Let $X\subset\p^{2n+1}$ be a smooth, linearly normal, Bronowski variety of dimension $n$ presenting the hyperplane case. Then $X$ is projectively equivalent either to a rational normal scroll or to the Segre embedding of $\p^1\times \Q^{n-1}$ or to one of the four lagrangian
grassmannians.
\end{corollary}

This can be deduced from Jordan--von Neumann--Wigner classification theorem of simple, complex, unitary Jordan algebras. However in \cite[Theorem 5.7]{PR} the same result is proved only with geometric arguments relying on 
Ein--Shepherd-Barron's classification of $(2,2)$ \emph{special} Cremona transformations, i.e.  
$(2,2)$  Cremona transformations whose base locus is smooth and irreducible (see \cite {esb}).
Hence  \cite[Theorem 5.7]{PR}, or  Ein--Shepherd-Barron's classification in  \cite {esb}, 
are equivalent to the aforementioned classification of simple, complex, unitary Jordan algebras. 

\begin{remark}
As far as we know Corollary \ref {cor:smooth} is the first classification result of a class of smooth OADP
varieties in arbitrary dimension. From a hierarchical viewpoint, OADP varieties presenting the hyperplane
case are the less complicated ones and therefore one may expect that they form the class
with more examples. 

Divisors of type $(1,2)$ in the Segre variety ${\rm Seg}(1,n)\subset\p^{2n+1}$ provide an infinite series of examples of smooth OADP varieties presenting {\it the hyperquadric case}, i.e. the fundamental hypersurface $V\subset\p^n$ is a 
quadric (which is smooth in this case). 
We are not aware of any other series of smooth examples presenting the \emph{degree $d$--case} (i.e. the fundamental hypersurface has degree $d$), with $d\ge 3$. 
\end{remark}

\subsubsection{Dimensions 3 and 4}\label{ssec:3-4}

In \cite[\S\S 4, 5]{PR} there is the classification of irreducible $n$--dimensional varieties $X\subset\p^{2n+1}$ which are 3--covered by twisted cubics for $n=2,3,4$. The same method of proof could be extended  to the case $n=5$.  In view of \S \ref {Jalg},  the case $n=3$ boils down to the well known 
classification of planar quadratic transformations. The result, which can be also directly 
deduced from
Theorem \ref {th:hyperplanar}, is as follows:

 \begin{corollary}\label{prop:planar} Let $X\subset \PP^7$ be a normal, geometrically normal, Bronowski  threefold  presenting case $(H2)$. Then $X$ is the image in $\PP^ 7$ of $\PP^ 3$ via the rational map determined by a linear system of cubic surfaces singular along a 0--dimensional curvilinear scheme $Z$ of length 3. If  $Z$ consists of three distinct points, then $X={\rm Seg}(1^3)$ is the Segre embedding of $\p^1\times\p^1\times\p^1\cong \p^1\times \Q^2$; if $Z$ is supported at  two distinct points, then $X$ is the image via the Segre embedding of $\PP^ 1\times S(0,2)$ {\rm(}$S(0,2)$ is the quadric cone in $\p^3${\rm)}. 
 
 Equivalently, $X$ is  projectively equivalent to a {twisted cubic curve over a 3--dimensional cubic complex, unitary Jordan algebra  $\AAA$}, where $\AAA$ is isomorphic to one of the following: 
 \[
{\text either}\quad  \C\times \C\times \C\quad {\text or}\quad \C\times \frac {\C[x]}{(x^ 2)}\quad {\text or}\quad  \frac {\C[x]}{(x^ 3)}.
 \] 
 \end{corollary}

\begin{remark} {\rm The three varieties in the statement of Corollary  \ref {prop:planar} appear in Fujita's classification
of varieties with $\Delta$--genus one. They fill up a component 
of the Hilbert scheme, whose general member corresponds to ${\rm Seg}(1^3)$. The 
other two varieties have a double line. They are Verra varieties. 
By projecting form the double line one finds $S(2,2)$ and $S(1,3)$ in the two cases.} 
\end{remark}

One can treat also the four  dimensional case, whose analysis depends on the classification of  quadro--quadric
Cremona transformations in $\p^3$ (see \cite{PRV}) and on the classification of four dimensional, complex, cubic Jordan algebras over the complex field  (see again \cite[\S\S 4,5]{PR} for details).

\begin{corollary}\label{cor:fourdim} 
 Let $X\subset \PP^9$ be a normal, geometrically normal, Bronowski  fourfold  presenting case $(H2)$. Then $X$ is projectively equivalent to
a  \emph{twisted cubic curve} over a 4--dimensional, cubic, complex, unitary Jordan algebra $\AAA$, where $\AAA$ is isomorphic to one of the seven cases in \cite[Table 2]{PR}.
 \end{corollary}

\begin{remark}{\rm 
In \cite{Semple} (see also \cite{brunoverra}) Semple presents the classification of general $(2,2)$ Cremona transformations
of $\p^4$.  {The classification of every $(2,2)$ Cremona transformation in $\p^4$ has been recently completed by Luc Pirio and the second author. This result
implies, among other things,  the classification of normal, geometrically normal, Bronowski  fivefolds  presenting case $(H2)$. We do not state this classification here.}}  
\end{remark}


\begin{thebibliography}{H}

\bibitem {Br} J. Bronowski, The sum of powers as canonical
expressions, \textit{Proc. Cam. Phil. Soc.} {\bf 29}
(1932), 69--82.

\bibitem{brunoverra} A.~Bruno, A.~Verra, 
On quadro-quadric Cremona transformations of $\p^4$,  pre--print 2009.


\bibitem {WDV} L.~Chiantini, C.~Ciliberto, Weakly defective varieties, \textit{Trans.
A.M.S.} \textbf{354} (2002), 151--178.

\bibitem {CLM} C.~Ciliberto, A.~Lopez, R.~Miranda, Some remarks on the obstructedness of cones over curves of low genus, in \textit{Higher dimensional complex varieties,
Proceedings Trento 1994}, De Gruyter, Berlin, 1996, 167--182.

\bibitem{OADP} C.~Ciliberto, M.~Mella, F.~Russo, Varieties with one apparent double point, \textit{J. Algebraic Geom.}  \textbf{13} (2004), 475--512.

\bibitem{CR} C. 
 Ciliberto, F.  Russo, 
Varieties with minimal secant degree and linear systems of maximal dimension on surfaces, 
\textit{Adv. Math.} {\bf 200}  (2006),  1--50.

\bibitem {crs} C.~Ciliberto, F.~Russo, A.~Simis, Homaloidal hypersurfaces and hypersurfaces with vanishing Hessian, \textit{Adv. Math.} \textbf{218} (2008), 1759--1805.

\bibitem {cs} C.~Ciliberto, E.~Sernesi, Singularities of the
theta divisor and congruences of planes,
\textit{J. Algebraic Geom.} {\bf 1} (1992), 231--250.

\bibitem{dP} P.~del Pezzo, Sugli spazi tangenti a una superficie o a una variet\` a 
immersa in uno spazio a pi\ `u dimensioni, {\it Rend. Acc. Napoli} {\bf 25} (1886), 176--180. 

 \bibitem{Ein} L.~Ein, Varieties with small dual variety. I,
{\it Invent.  Math.}   {\bf 86} (1986), 63--74.

 \bibitem{esb} L.~Ein, N.~Shepherd--Barron, Some special Cremona transformations, 
{\it Am. J.  Math.}   {\bf 111} (1989), 783--800.


\bibitem {fu1} T.~Fujita, Classification of projective varieties of $\Delta$--genus one, \textit{Proc. Japan Acad. Ser. A Math. Sci.} \textbf{58} (1982), 113--116.

\bibitem {fu2} T.~Fujita, Projective varieties of $\Delta$--genus one, in \textit{Algebraic and topological theories}, Kinokunia, Tokyo, 1986, 149--175.




\bibitem {GH} P.~Griffiths, J.~Harris, {Algebraic geometry and local differential geometry}, \textit{Ann.
Scient. Ec. Norm. Sup.} {\bf 12} (1979), 355--432.


\bibitem {H} J.~Harris, {Algebraic geometry}, \textit{Graduate Texts in Math.} 133, Springer--Verlag, 1992.

\bibitem{IR} P. Ionescu, F. Russo,  Varieties with quadratic entry locus. II, {\it Compos. Math.}  {\bf 144} (2008), 949--962.

\bibitem {LV} R.~Lazarsfeld, A.~van de Ven, {Topics in the geometry of
projective space}, \textit{DMV Seminar} {\bf 4}, Birkh\" auser, 1984.

\bibitem{mccrimmon} K.~Mc Crimmon,  {A Taste of Jordan Algebras},  \textit{Universitext}, Springer--Verlag, 2004. 

\bibitem{Mu2} S.~Mukai, {Simple Lie algebra and Legendre variety}, \textit{preprint},  1998,  http://www.math.nagoya-u.ac.jp/\~{}mukai 

\bibitem{PRV} I.~Pan, F.~Ronga, T.~Vust, {Transformations birationelles 
quadratiques de l'espace projectif complexe \` a trois dimensions}, \textit{Ann. Inst. Fourier Grenoble}  {\bf 51} (2001), 1153--1187.

\bibitem {PR}  L.~Pirio, F.~Russo, {On projective varieties $n$--covered by curves of degree $\delta$}, \textit{pre--print}  2010. 

\bibitem{PR2}  L. ~Pirio, F.~Russo, {Extremal varieties 3-rationally connected by cubics, quadro-quadric Cremona transformations and rank 3 Jordan algebras}, \textit{preprint}, 2011, submitted for publication.


\bibitem {Rog} E.~Rogora, Varieties with many lines, \textit{Manuscripta Math.} \textbf{82} (1994), 207--226. 

\bibitem {Ru} F.~Russo, {On a theorem of Severi}, \textit{Math. Ann.} {\bf
316} (2000), 1--17.



\bibitem{Semple}  J.~G.~Semple, {Cremona Transformations of Space of Four Dimensions by means of
Quadrics, and the Reverse Transformations}, \textit{Philosophical Transactions of the Royal Society of London}, Series A, {\bf 228} (1929), 331--376.


\bibitem {Se} F.~Severi, {Intorno ai punti doppi impropri
di una superficie generale dello spazio a quattro dimensioni e ai
suoi punti tripli apparenti}, \textit{Rend. Circ. Mat. Palermo}  {\bf 15}
(1901), 33--51.

\bibitem {violo}  M.~G.~Violo, {Variet\`a con un punto doppio
apparente}, \textit{Tesi di Dottorato}, Torino, 1997.

\bibitem{Zak2} F.~L.~Zak, {Varieties of small codimension arising from group actions}, Addendum to:
\cite{LV}.
\end{thebibliography}
\end{document}